\documentclass{article}

\usepackage{amsmath,amsfonts,amssymb,latexsym,amsthm,verbatim,a4wide}

 \theoremstyle{plain}
  \newtheorem{thm}{Theorem}[section]
  
  \newtheorem{prop}[thm]{Proposition}
  
  \newtheorem{defi}[thm]{Definition}

\newcommand{\tn}{\widetilde{\nabla}_{n} }
\newcommand{\Z}{{\mathbb{Z}}}
\newcommand{\re}{{\mathbb{R}}}
\newcommand{\vc}[1]{{\mathbf #1}}
\newcommand{\blah}[1]{}
\newcommand{\Ent}{{\rm Ent}}

\begin{document}

\title{A natural derivative on $[0,n]$ and a binomial Poincar\'{e} inequality}
\author{Erwan~\textsc{Hillion}\thanks{
Institut de Math\'{e}matiques,
Rue Emile-Argand 11,
Case Postale 2,
CH-2007 Neuch\^{a}tel Suisse. Email {\tt erwan.hillion@unine.ch}}
\and
 Oliver~\textsc{Johnson}\thanks{Statistics Group, Department of Mathematics, 
University of Bristol, University Walk, Bristol, BS8 1TW, UK. 
Email {\tt o.johnson@bris.ac.uk}} 
\and Yaming~\textsc{Yu}\thanks{Department of Statistics, University 
of California, Irvine, CA 92697, USA. Email {\tt yamingy@uci.edu}}}
\date{\today}

\maketitle

\begin{abstract} We consider probability measures supported on
a finite discrete interval $[0,n]$. We introduce a new finite
difference operator $\nabla_n$, defined as a linear combination of left
and right finite differences. We show that this operator $\nabla_n$ plays a key role in a 
new Poincar\'{e} (spectral gap) inequality with respect to binomial weights, with the
orthogonal Krawtchouk polynomials acting as eigenfunctions of the relevant
operator. We briefly discuss the relationship of this operator to the problem of
optimal transport of probability measures.
\end{abstract}


\section{Introduction and main results}

Many results in functional analysis are better understood in the context
of continuous spaces than discrete.
One reason that the real-valued case is more tractable than integer-valued
problems is the existence
of a spatial derivative $\frac{\partial}{\partial x}$,  well-defined
in the sense that the left and right derivatives coincide for a large
class of functions. However, the 
situation is more complicated for integer-valued functions $f$. There
exist two competing derivatives $\nabla^l$ and $\nabla^r$,  
defined as $\nabla^l f(k) = f(k) - f(k-1)$ and $\nabla^r f(k) = f(k+1) - f(k)$,
which are adjoint with respect to
counting measure on $\Z$.
In this paper, we define a new finite difference operator for
functions on $[0,n]$, which interpolates
between $\nabla^l$  and $\nabla^r$.
\begin{defi} \label{def:nabla}
Fix an integer $n \geq 1$, and denote by 
$\nabla_n$ the finite difference operator defined by 
\begin{eqnarray}(\nabla_n f) (k) & = &
\frac{k}{n} (\nabla^l f)(k) + \frac{n-k}{n} (\nabla^r f)(k) \nonumber \\
& = & \frac{k}{n} 
(f(k)-f(k-1))+\frac{n-k}{n} (f(k+1)-f(k))
\label{NablaDef} .\end{eqnarray}
\end{defi}
We will argue that this operator has certain desirable properties,
and as such deserves further attention. In particular, we will
show that in two senses 
it is a natural choice of derivative in relation to
binomial measures $b_{n,t}(k) = \binom{n}{k} t^k (1-t)^{n-k}$.

Firstly, in Section \ref{sec:translate}, we will show that this operator
$\nabla_n$ acts like the translation operator on the real line. That is, in Equation
(\ref{HeuriEDP}) below, we describe how a probability measure $\mu$ on $\re$ can
be smoothly translated using a sequence of intermediate measures $\mu_t$. Equation (\ref{HeuriEDP})
describes the effect of this translation action through its effect on arbitrary test functions
$f$.
 We prove the following theorem, which acts as a discrete counterpart of (\ref{HeuriEDP}), with
the relationship between measure $b_{n,t}$ and operator $\nabla_n$ playing a key role:
\begin{thm} \label{thm:translate}
The operator $\nabla_n$ gives a smooth translation of 
point masses from point 0 to point $n$ using the binomial measures $b_{n,t}$
in that
\begin{enumerate}
\item
$b_{n,t}$ satisfies the initial condition $b_{n,0} = \delta_0$ and the 
final condition $b_{n,1} = \delta_n.$ 
\item For every function $f: \Z \rightarrow \re$, 
\begin{equation}\label{EqTranslation2} 
\frac{\partial}{\partial t} \sum_{k \in \Z}    f(k) b_{n,t}(k)
 = n \sum_{k \in \Z} (\nabla_n f)(k) b_{n,t}(k) .\end{equation}
\end{enumerate}
\end{thm}

Secondly, in Proposition \ref{prop:ladder} below we
 will show that the map $\nabla_n$ and its adjoint $\tn$ (with respect to
binomial weights)
act as ladder operators for the Krawtchouk polynomials $\phi_r$ (see Theorem
\ref{thm:krawt}). 
This allows us to describe the spectrum of the map
$\left( \tn \circ \nabla_n \right)$, with
 $\phi_r$ being eigenfunctions with eigenvalue $\frac{r(n-r+1)}{n^2 t(1-t)}$.
In particular, taking the smallest non-zero eigenvalue leads to a Poincar\'e (spectral
gap)
inequality for the binomial 
law, using the natural derivative operator $\nabla_n$, and gives the case of
equality.
\begin{thm}\label{PoincBino}
Fix $t \in (0,1)$ and consider function $f : \{0,\ldots n\} \rightarrow \re$ 
satisfying $\sum_{k=0}^n f(k) b_{n,t}(k)=0.$ Then
\begin{equation} \label{eq:eigen} 
  \sum_{k=0}^n b_{n,t}(k) f(k)^2 
\leq n t(1-t) \sum_{k=0}^n b_{n,t}(k) \left( \nabla_n f(k) \right)^2. \end{equation} 
Equality holds if and only if $f$ is a linear combination of 
$\phi_1(k) = \frac{1}{1-t}(k-nt)$ and 
$\phi_n(k)=n!\left(\frac{-t}{1-t}\right)^{n-k}.$
\end{thm}

The idea of studying Poincar\'{e} inequalities with respect to discrete
distributions is not a new one.
For example, Bobkov and co-authors
\cite{bobkov2, bobkov5, bobkov, bobkov3}
give results concerning probability measures supported on
the discrete cube (with the difference $\nabla^r$
taken modulo 2). Cacoullos
\cite{cacoullos}, Chen and Lou \cite{chen5} and Klaasen 
\cite{klaasen} give results concerning
$\nabla^r$ on
$\Z$ and $\Z^n$. 
In particular, Table 2.1 of Klaassen \cite{klaasen} shows
that for Poisson mass function $\Pi_\lambda$, if $\sum_k f(k) \Pi_\lambda(k) = 
0$ then
\begin{equation} \label{eq:poincpoi} 
 \sum_{k=0}^{\infty} \Pi_{\lambda}(k) f(k)^2
\leq \lambda \sum_{k=0}^{\infty} \Pi_{\lambda}(k) \left( \nabla^r f(k) \right)^2.\end{equation}
This can be understood as a consequence of the fact that $\nabla^r$ (and its
adjoint with respect to Poisson weights $\widetilde{\nabla}^r$) act as ladder operators 
with respect to Poisson-Charlier polynomials, meaning that the Poisson-Charlier
polynomials are eigenfunctions of $\left( \widetilde{\nabla}^r \circ \nabla^r
\right).$
These results also have an analogy with the
work of Chernoff \cite{chernoff}, where the corresponding result was proved
for normal random variables, with the Hermite polynomials acting as
eigenfunctions of the corresponding map.

However, Klaassen does not deduce such a clean result for binomial weights,
requiring a weighting term on the right-hand side
\begin{equation} \label{eq:poincbin}
\sum_{k=0}^n b_{n,t}(k) f(k)^2 \leq t \sum_{k=0}^n b_{n,t}(k)(n-k) \left( \nabla^r f(k) \right)^2  \end{equation} 
We can summarise the difference between our Theorem \ref{PoincBino} and Klaassen's
Equation (\ref{eq:poincbin}) by saying that we have altered the definition of
the derivative, whereas Klaassen altered the binomial distribution in
question. Note that as $n \rightarrow \infty$ with $t n = \lambda$, 
Theorem \ref{PoincBino} converges to Equation (\ref{eq:poincpoi}).

Note that although we do not directly discuss applications here, in other settings
 the rate of convergence in
variance of reversible Markov chains can be bounded in terms of the spectral
gap (see for example \cite[Lemma 2.1.4]{saloff}). 

In general, Poincar\'{e} inequalities are often viewed as a consequence of 
log-Sobolev inequalities (see for example \cite[Lemma 2.2.2]{saloff}). In
particular, for Poisson measures $\Pi_\lambda$, Bobkov and Ledoux 
\cite[Corollary 4]{bobkov3}
prove that for any positive function $f$,
\begin{equation} \label{eq:logsob} \Ent_{\Pi_\lambda}(f) \leq
\lambda \sum_{k=0}^{\infty} \Pi_{\lambda}(k) \frac{\left( \nabla^r f(k) \right)^2}
{f(k)}, \end{equation}
and show that Klaasen's Poincar\'{e} inequality (\ref{eq:poincpoi}) can be deduced
from (\ref{eq:logsob}).  Here, $\Ent_{\nu}(f) = \sum_k \Theta(f(k)) \nu(k) -
\Theta \left( \sum_k f(k) \nu(k) \right)$, where $\Theta(t) = t \log t$.
It is natural to conjecture that an equivalent of Equation (\ref{eq:logsob})
should hold for Binomial random variables with our natural derivative $\nabla_n$, that
is
\begin{equation} \label{eq:binlogsob} \Ent_{b_{n,t}}(f) \leq
nt(1-t) \sum_{k=0}^{n} b_{n,t}(k) \frac{\left( \nabla_n f(k) \right)^2}
{f(k)}. \end{equation}
However, this result (\ref{eq:binlogsob}) is in general false. Consider for
example
$n=2$, $t=1/2$, $f(0) = f(2) = 9/10$, $f(1) = 1/10$. In this case,
$\Ent_{b_{n,t}}(f)   = 0.18403$ and the right-hand side of Equation (\ref{eq:binlogsob})
is $0.17777$, and the inequality fails. The question of natural conditions on $f$ under
which Equation (\ref{eq:binlogsob}) holds remains open.

The structure of the remainder of the paper is as follows. In
Section \ref{sec:translate}, we discuss the translation problem in $\Z$
and prove the existence of a fundamental solution for
the problem under the choice of $\nabla$ as the
$\nabla_n$ from Definition \ref{def:nabla}. 
In Section \ref{sec:poincare} we prove Proposition \ref{prop:ladder},
the key result leading to the Poincar\'{e} inequality Theorem
\ref{PoincBino}.

\section{The translation problem in $\Z$} \label{sec:translate}

It is clear that there exists an unambiguous definition of translations of
real-valued probability measures, defined as the push-forward of the translation
map. That is, let 
 $\mu$ be a probability measure on $\re$ (with its Borel $\sigma$-algebra) 
having a smooth density $\rho$ w.r.t. the Lebesgue measure $dx$. The 
$n$-translation of $\mu$, 
where $n \in \re$, is the family of measures 
$(\mu_t=\rho_t dx)_{t\in [0,1]}$, where the density $\rho_t$ is defined by 
\begin{equation}\label{HeuriPF}\forall x \in \re, \ \rho_t(x) = \rho(x-nt).
\end{equation} 
In other words, the measure $\mu_t$ is the push-forward of $\mu$ by the
translation map $T_t(x)=x+nt=(1-t)x+t(x+n)$. In particular, 
\begin{equation}\label{HeuriEDP0}\frac{\partial}{\partial t} 
\rho_t(x) = -n \frac{\partial}{\partial x} \rho_t(x).\end{equation} 
This can be generalized for non absolutely continuous probability measures, 
writing Equation (\ref{HeuriEDP0}) in the sense of distributions: 
\begin{equation}\label{HeuriEDP}
\frac{\partial}{\partial t}\ \int_\re f(x) d \mu_t(x) = 
n \int_\re \frac{\partial}{\partial x} f(x) d \mu_t(x),
\mbox{\;\;\;\; for all $f \in \mathcal{C}_c^\infty (\re)$.}
\end{equation}
This equation means that the measure $\mu_t$ is the convolution of the initial 
measure $\mu_0$ with the fundamental solution of Equation (\ref{HeuriEDP}):
\begin{equation}\label{HeuriConvo}\mu_t=\mu_0 * \delta(x-nt).\end{equation} 
Notice that this construction of $\mu_t$ allows a smooth interpolation of
probability measures. In this paper we generalize 
these heuristics to the case of probability measures on $\Z$. 

\begin{defi}
A probability measure $\mu_1$ on $\Z$ is the $n$-translation of another 
probability measure $\mu_0$ if $$ \mu_1(k+n)= \mu_0(k)
\mbox{ \;\;\;\; for all $k \in \Z$}.$$
In particular, we will consider measures that smoothly interpolate between
point masses
\begin{equation} \mu_0 = \delta_0 \mbox{\;\;\; and \;\;\;} \mu_1 = \delta_n.
\label{eq:initial}
\end{equation} 
\end{defi}
The 
non-connectedness of $\Z$ makes it impossible to generalize 
Equation (\ref{HeuriPF}) directly.
However, we will adapt the ``PDE point of view'', given in 
Equation (\ref{HeuriEDP}), to construct the $n$-translation 
of point masses (\ref{eq:initial}),  in a way that satisfies
\begin{equation}\label{EqTranslation} \frac{\partial}{\partial t}
\sum_{k \in \Z}    f(k) \mu_t(k) = 
n \sum_{k \in \Z}  \nabla f(k) \mu_t(k).\end{equation}
The main problem in this adaptation is 
to find the correct derivative operator $\nabla$ on $\Z$. 
In general, we make the following
definition:

\begin{defi}
A spatial derivative $\nabla$ on $\Z$ is a linear operator in the space 
of functions on $\Z$ that maps any function $f$ to another function $\nabla f$, 
where, for each $k \in \Z$, there exists a coefficient $\alpha_k \in [0,1]$ 
such that $$(\nabla f)(k) = \alpha_k 
(\nabla^l f)(k) + (1-\alpha_k) (\nabla^r f)(k).$$ 
\end{defi}

In other words, a derivative is defined by a family of coefficients 
$(\alpha_k \in [0,1])$, for $k \in \Z$. Each of these coefficients tells us 
how to mix, at a given point $k$, left and right derivatives. For example, 
the left (resp. right) derivative 
corresponds to the case where all the coefficients are equal to $1$ (resp. $0$).

First we show that a spatial derivative on $\Z$ for which there 
exists a fundamental solution to the $n$-translation problem must follow 
some necessary conditions. We next show that these necessary conditions allow 
us to reduce the translation problem to a more understandable problem of 
linear algebra in finite dimensions.

\begin{prop}\label{AlphaNecessaire}
Fix integer $n\geq 1$ and a derivative $\nabla$ on $\Z$ defined by 
a family of coefficients $(\alpha_k)_{k \in \Z}$. If there 
exists a solution $\mu_t$ to the $n$-translation problem 
(\ref{eq:initial}), (\ref{EqTranslation}) 
associated with 
$\nabla$ then $\alpha_0 = 0$ and $\alpha_n =1$.
Moreover, the support of $\mu_t$ is contained in $\{0,\ldots n\}$.
\end{prop}
\begin{proof}
 Let us first consider the function $f : \Z \rightarrow \re$ defined by 
$f(k)=0$ if $k<0$, and $f(k)=1$ if $k \geq 0$. It is easy to show that 
$(\nabla f)(-1) = 1-\alpha_{-1}$, $(\nabla f)(0) = \alpha_0$, and 
$(\nabla f)(k)=0$ elsewhere. \\
Let us now define the function $g : [0,1] \rightarrow \re$ by 
$$g(t):=\sum_{k \in \Z} f(k) \mu_t(k) := \sum_{k \geq 0} \mu_t(k).$$ 
The initial and final conditions satisfied by $\mu_t$ show that $g(0)=1=g(1)$. 
On the other hand, the Equation~\eqref{EqTranslation2} shows that 
$$g'(t) = n \sum_{k \in \Z} \mu_t(k) \nabla f(k) = n [(1-\alpha_{-1})\mu_t(-1)
+\alpha_0 \mu_t(0)].$$ 
In particular $g'(t) \geq 0$. The fact that $g(0)=g(1)$ thus implies 
that $g'(t)=0$ for every $t\in[0,1]$, and the condition $g'(0)=0$ can be 
written $\alpha_0=0$. Moreover, the fact that $g(t)=1$ for every $t\in [0,1]$
 implies $$\sum_{k \geq 0} \mu_t(k) = 1, $$ so $\mu_t$ is supported on $\Z_+$.

If we apply the same arguments to the function $f$ defined by $f(k)=1$ 
if $k\leq n$, and $f(k)=0$ if $k>n$, we find that $\alpha_n=1$, and that 
$\mu_t$ is supported on $\{k \in \Z \ | \ k\leq n\}$. 
\end{proof}

An interesting consequence of Proposition~\ref{AlphaNecessaire} is that 
the translation problem of Equation (\ref{EqTranslation}) can be 
restricted to $\mu_t$ supported on $[0,n]$. That is, we can replace (\ref{EqTranslation})
by 
\begin{equation}\label{EqTransla} 
\frac{\partial}{\partial t} \sum_{k=0}^n    f(k) \mu_t(k)
= n \sum_{k=0}^n  \nabla f(k) \mu_t(k).\end{equation}

Now, let us consider the canonical basis $\mathcal{CB} := (e_0,\ldots e_n)$ of 
the linear space of functions $\{0,\ldots, n\} \rightarrow \re$. Let $X(t)$ 
be the column vector representing $\mu_t$ in $\mathcal{CB}$ 
(probability measures are canonically identified with functions), ie for 
every $k \in \{0,\ldots n\}$, $(X(t))_k := \mu_t(k)$. The initial (resp. final) 
condition $\mu_0 =\delta_0$ (resp. $\mu_1=\delta_n$) is equivalent to $X(0)=e_0$ 
(resp. $X(1)=e_n$). Moreover, Equation~\eqref{EqTransla} is equivalent to the
fact that for all vectors  $ Y \in M_{n,1}(\re)$
\begin{equation} \langle X'(t),Y \rangle = \frac{\partial}{\partial t} \langle X(t), Y
\rangle = 
 n \langle X(t) , \nabla  \rangle = n \langle \nabla^* X(t), Y \rangle, \end{equation} 
where $\langle .,. \rangle$ is the usual (unweighted) scalar product on column vectors, and
where $\nabla^*$ represents the adjoint with respect to this scalar product.
This allows us to deduce that
\begin{equation}\label{EqTranslaMatrice} X'(t) = n \nabla^* X(t),\end{equation} 
and basic theorems on first-order linear differential systems thus allow us to 
write the $n$-translation problem:

\begin{thm}
Let $n \geq 1$ be an integer, and $\nabla$ be a derivative on $\Z$, with 
$\alpha_0=0$ and $\alpha_n=1$. Let $A_\nabla$ be the matrix associated with
$\nabla$ and $n$. There exists a fundamental solution to the $n$-translation 
problem associated with $\nabla$ if and only if, for every $t\in[0,1]$, the 
column matrix $$X(t):=\exp(ntA_\nabla) e_0$$ has all its coefficients 
non-negative, and satisfies the final condition 
\begin{equation}\label{ExpCond}X(1)=e_n.\end{equation}
The fundamental solution $\mu_t(k)$ is then given by $\mu_t(k)=(X(t))_k$.
\end{thm}

We prove Theorem \ref{thm:translate} using the properties of the 
spatial derivative $\nabla_n$ introduced in Definition \ref{def:nabla}. 
In this case we can be explicit about the form of $\nabla^*_n$, and introduce
a further map $\tn$ which will be used
 to prove Theorem \ref{thm:translate} and the  Poincar\'{e} inequality Theorem \ref{PoincBino}.
\begin{defi} \label{def:furthermaps} \mbox{ }
\begin{enumerate}
\item
Let $\nabla^*_n$ be the adjoint operator of $\nabla_n$ for the unweighted scalar 
product on $l^2(\{0,\ldots n\})$. We have the formula 
$$\nabla^*_ng(k) = \frac{1}{n} \left( (n-k+1) g(k-1)-(n-2k)g(k)-(k+1)g(k+1) 
\right),$$ 
where $g(-1)=g(n+1)=0$.
\item
We now fix $t \in (0,1)$.
 Let $\tn$ be the adjoint operator of $\nabla_n$ for the scalar 
product with respect to the binomial law $b_{n,t}$ 
(taking $t \notin \{0,1\}$ ensures that it is truly a scalar product 
on the space of functions $\{0,\ldots n\} \rightarrow \re$)). 
We have:
\begin{eqnarray}
\tn f (k) &=& \frac{1}{b_{n,t}(k)} \nabla_n^*(f(k)b_{n,t}(k)) \nonumber \\
&=& \frac{n-k+1}{n} \frac{b_{n,t}(k-1)}{b_{n,t}(k)} f(k-1)-\frac{n-2k}{n}f(k)
-\frac{k+1}{n}\frac{b_{n,t}(k+1)}{b_{n,t}(k)}f(k+1) \nonumber \\
& = & \frac{k}{n} \frac{1-t}{t} f(k-1)-\frac{n-2k}{n} f(k) 
- \frac{n-k}{n} \frac{t}{1-t} f(k+1). \label{TildeNablaDef}
\end{eqnarray}
\end{enumerate}
\end{defi}
The equivalence of the last two results follows since for all $k$,
\begin{eqnarray*}
\frac{b_{n,t}(k-1)}{b_{n,t}(k)} 
&=& \frac{k}{n-k+1} \frac{1-t}{t}.
\end{eqnarray*}
We can relate properties of $\tn$ and $\nabla^*_n$ using conjugation by
the linear operator $D$ that maps any function 
$f : \{0,\ldots, n\} \rightarrow \re$ to the function $Df$ 
defined by $$\forall k \in \{0,\ldots, n\}, \ Df(k)=b_{n,t}(k)f(k).$$ 
Moreover, as $t \in (0,1)$, $D$ is invertible and 
$$\forall k \in \{0,\ldots, n\}, \ D^{-1}f(k) = \frac{1}{b_{n,t}(k)}f(k).$$ 
This operator is useful to give a very simple relation between $\nabla^*_n$ and
 $\tn$: 
\begin{equation} \label{eq:conjugate} \tn = D^{-1} 
\circ \nabla^*_n \circ D.\end{equation}

\begin{proof}[Proof of Theorem \ref{thm:translate}]
We simply verify that  (\ref{EqTranslaMatrice}) holds taking 
$X(t) = b_{n,t}(k)$ and $\nabla^*$ in the form given by
Definition \ref{def:furthermaps}. We observe that in this
case both sides of (\ref{EqTranslaMatrice}) have $k$th component
equal to $b_{n,t}(k) \left( k/t - (n-k)/(1-t) \right)$. The fact that $\frac{\partial}{\partial t}
b_{n,t}(k)$ takes this form is immediate, and the corresponding result for
the right hand side follows by Equations (\ref{TildeNablaDef}) and (\ref{eq:conjugate}) since
$n \frac{1}{b_{n,t}(k)} \nabla^*_n b_{n,t}(k) = n
\tn \vc{1} = k/t - (n-k)/(1-t)$, where $\vc{1}$ denotes the function which is identically 1. 
 \end{proof}

\section{Proof of the Poincar\'{e} inequality} \label{sec:poincare}

From now on, we fix an integer $n \geq 1$, and we denote by 
$\nabla_n$ the finite difference operator of Definition \ref{def:nabla}. 
We recall the definition of the Krawtchouk polynomials from \cite{szego}.
\begin{thm} \label{thm:krawt}
There exists a basis of polynomials in $k$, 
denoted $\phi_0,\ldots, \phi_n$, ``laddered'' 
(i.e. with $deg(\phi_r)=r$), and such that 
\begin{equation} \label{eq:krawtnorm}
\sum_{k=0}^n \phi_r(k) \phi_{s}(k) b_{n,t}(k) = \frac{n! r!}{(n-r)!} \left( \frac{t}{1-t} \right)^r \delta_{rs}
:= C_{n,r} \delta_{rs}.
\end{equation}
This family of polynomials is uniquely determined by the generating function in
 $w$  \begin{equation}
\label{GenerYoung}
P(k,w) := \sum_{r=0}^n \frac{(1-t)^r}{r!} \phi_r(k) w^r = (1+(1-t)w)^k(1-tw)^{n-k}.
\end{equation}
\end{thm}
The discrete derivatives in $k$ of $P(k,w)$ can be obtained by using the 
formulas 
\begin{eqnarray} P(k-1,w) & = & P(k,w) \frac{1-tw}{1+(1-t)w}
\mbox{ for all $k \geq 1$} \label{eq:discA} \\
P(k+1,w) & = & P(k,w)\frac{1+(1-t)w}{1-tw} \mbox{ for all $k \leq n-1$}
\label{eq:discB}
\end{eqnarray}
%
Finally, since $\frac{\partial}{\partial w} w^r=rw^{r-1}$, we obtain 
\begin{equation} \label{eq:discD}
\sum_{r=0}^n \frac{(1-t)^r}{r!} r \phi_r(k) w^r =
 w \frac{\partial}{\partial w} P(k,w)
= w P(k,w)  
\left(\frac{(1-t)k}{1+(1-t)w}-\frac{t(n-k)}{1-tw}\right).\end{equation}
Notice that $\phi_0$ is the function identically equal to $1$, and so 
$\nabla_n \phi_0 = 0$, 
which gives a sense to Proposition \ref{prop:ladder} when $r=0$. 
To simplify the proof, we will define $\phi_{-1}= \phi_{n+1} = 0$.
\begin{prop} \label{prop:ladder} 
For every $r \in \{0,\ldots, n \}$, we have
\begin{enumerate}
\item \label{NablaPhi}
The operator $\nabla_n$ maps $\phi_r$ to a multiple of $\phi_{r-1}$: 
$ \begin{displaystyle}
\nabla_n \phi_r = \frac{r(n-r+1)}{n(1-t)} \phi_{r-1}.\end{displaystyle}
$
\item
\label{TildeNablaPhi}
The operator $\tn$ maps $\phi_r$ to a multiple of $\phi_{r+1}$:
$ \begin{displaystyle} \tn \phi_r = \frac{1}{nt} \phi_{r+1}.
\end{displaystyle}$ 
\item \label{PhiVP} The Krawtchouk
polynomials are eigenfunctions for the linear map 
$\left( \tn \circ \nabla_n \right)$: 
$$ \left( \tn \circ \nabla_n \right) \phi_r 
= \frac{r(n-r+1)}{n^2t(1-t)} \phi_r.$$
\end{enumerate}
\end{prop}
Remark that these eigenvalues are not distinct, which does not allows us
to deduce directly that the family $(\phi_0,\ldots, \phi_n)$ is a basis of the 
space of 
functions $\{0,\ldots, n\} \rightarrow \re$. This fact comes from the 
orthogonality with respect to the binomial scalar product.

\begin{proof}[{\bf Proof of Proposition \ref{prop:ladder}}]
{\bf Part \ref{NablaPhi}:}
It suffices to check the polynomial identity 
$$\sum_{r=0}^n \frac{(1-t)^r}{r!} \nabla_n \phi_r(k) w^r
= \sum_{r=0}^n \frac{(1-t)^r}{r!} \frac{r(n-r+1)}{n(1-t)} \phi_{r-1}(k) w^r.$$
We will use the formula~\eqref{GenerYoung} to 
express both side of the last equation 
in terms of the polynomial $P(k,w)$. First, we have by Equations 
(\ref{eq:discA}) and (\ref{eq:discB}) that
\begin{eqnarray*}
\sum_{r=0}^n \frac{(1-t)^r}{r!} \nabla_n \phi_r(k) w^r &=& \nabla P(k,w) \\
&=& \frac{P(k,w)}{n} \left(k\left(1-\frac{P(k-1,w)}{P(k,w)}\right)
+(n-k)\left(\frac{P(k+1,w)}{P(k,w)}-1\right)\right) \\
&=& \frac{P(k,w)}{n} \left(k\left(1-\frac{1-tw}{1+(1-t)w}\right)
+(n-k)\left(\frac{1+(1-t)w}{1-tw}-1\right)\right) \\
&=& \frac{P(k,w)}{n}w \left(\frac{k}{1+(1-t)w}+\frac{n-k}{1-tw}\right).
\end{eqnarray*}
For the right hand side, we have using (\ref{eq:discD}) that
\begin{eqnarray*}
\sum_{r=0}^n \frac{(1-t)^r}{r!} \frac{r(n-r+1)}{n(1-t)} \phi_{r-1}(k) w^r 
&=& \frac{w}{n} \sum_{r=0}^n \frac{(1-t)^r}{r!} (n-r) \phi_{r}(k) w^r \\
&=& \frac{P(k,w)}{n}w \left(n-w\left(\frac{(1-t)k}{1+(1-t)w}-
\frac{t(n-k)}{1-tw}\right)\right) \\
&=& \frac{P(k,w)}{n}w \left(k \left(1-\frac{(1-t)w}{1+(1-t)w}\right)
+(n-k)\left(1+\frac{tw}{1-tw}\right) \right)\\
&=& \frac{P(k,w)}{n}w \left(\frac{k}{1+(1-t)w}+\frac{n-k}{1-tw}\right),
\end{eqnarray*}
which gives the desired result. 

{\bf Part \ref{TildeNablaPhi}:}
It suffices to check the polynomial identity 
$$\sum_{r=0}^n \frac{(1-t)^r}{r!} \tn \phi_r(k) w^r
= \sum_{r=0}^n \frac{(1-t)^r}{r!} \frac{1}{nt} \phi_{r+1}(k) w^r.$$
Let us begin by studying the right hand side. Using the convention 
$\phi_{n+1}=0$, we have by (\ref{eq:discD})
\begin{eqnarray*}
\sum_{r=0}^n \frac{(1-t)^r}{r!} \frac{1}{nt} \phi_{r+1}(k) w^r 
&=& \frac{1}{nt(1-t)w} 
\sum_{r=0}^n \frac{(1-t)^{r+1}}{(r+1)!}(r+1) \phi_{r+1}(k) w^{r+1} \\
&=& \frac{1}{nt(1-t)w} \sum_{r=0}^{n} \frac{(1-t)^r}{r!}r \phi_r(k) w^r \\
&=& \frac{1}{nt(1-t)w}w \frac{\partial}{\partial w} P(k,w) \\
&=& \frac{1}{nt(1-t)} P(k,w) \left(\frac{(1-t)k}{1+(1-t)w}-
\frac{t(n-k)}{1-tw}\right).
\end{eqnarray*}
The left hand side can be written 
$$\sum_{r=0}^n \frac{(1-t)^r}{r!} \tn \phi_r(k) w^r = \tn P(k,w),$$ 
and we calculate using (\ref{eq:discA}) and (\ref{eq:discB}) that
\begin{eqnarray*}
\tn P(k,w)&=& P(k,w) \left(\frac{k}{n} \frac{1-t}{t} \frac{P(k-1,w)}{P(k,w)}
-\frac{n-2k}{n} - \frac{n-k}{n} \frac{t}{1-t} \frac{P(k+1,w)}{P(k,w)}\right)\\
&=& \frac{P(k,w)}{nt(1-t)} \left(k(1-t)^2\frac{1-tw}{1+(1-t)w} - (n-2k)t(1-t)
 -(n-k)t^2\frac{1+(1-t)w}{1-tw}\right)\\
&=& \frac{P(k,w)}{nt(1-t)} \left((1-t)k \left(\frac{(1-t)(1-tw)}{1+(1-t)w}+t\right)
-t(n-k)\left(\frac{t(1+(1-t)w)}{1-tw}+(1-t)\right)\right)\\
&=& \frac{1}{nt(1-t)} P(k,w) \left(\frac{(1-t)k}{1+(1-t)w}-
\frac{t(n-k)}{1-tw}\right),
\end{eqnarray*}
and the proof is complete.

{\bf Part \ref{PhiVP}:} follows directly by combining the two previous results.
\end{proof}

Similarly,
there is another way to prove Part \ref{TildeNablaPhi} of Proposition \ref{prop:ladder},
 using the properties of the exponential of the operator $\nabla^*_n$:
\begin{proof}[{\bf Alternative proof of  Proposition \ref{prop:ladder}, Part \ref{TildeNablaPhi}.}]
\begin{equation}\label{ExpNablaStar}\forall t \in [0,1], \ 
\exp(nt \nabla^*_n)(e_0)= (b_{n,t}(0),\ldots, b_{n,t}(n))^T.\end{equation} 
The equation~\eqref{GenerYoung} allows us to show that 
the required result  is equivalent to 
\begin{equation}\label{GenerTn}\exp(nt(1-t)w \tn)(\phi_0) 
= (1+(1-t)w)^k(1-tw)^{n-k}.\end{equation}

As $\phi_0=(1,\ldots,1)^T$, the equation~\eqref{ExpNablaStar}:
\begin{eqnarray*}
D\phi_0 &=& (b_{n,t}(0),\ldots,b_{n,t}(n))^T \\
&=& \exp(nt \nabla^*)(e_0).
\end{eqnarray*}

\begin{eqnarray*}
\exp(nt(1-t)w \tn)(\phi_0) &=& D^{-1} \exp(nt(1-t)w \nabla^*) D\phi_0 \\
&=& D^{-1} \exp(nt(1-t)w \nabla^*) \exp(nt \nabla^*)(e_0) \\
&=& D^{-1} \exp(nt(1+(1-t)w) \nabla^*)(e_0).
\end{eqnarray*}
This means that, for every $k \in \{0,\ldots,n\}$:
\begin{eqnarray*}
\exp(nt(1-t)w \tn)(\phi_0)(k) &=& \frac{1}{b_{n,t}(k)} b_{n,t(1+(1-t)w)}(k) \\
&=& \left(\frac{t(1+(1-t)w)}{t}\right)^k 
 \left(\frac{1-t(1+(1-t)w)}{1-t}\right)^{n-k} \\
&=& (1+(1-t)w)^k(1-tw)^{n-k}.
\end{eqnarray*}
This proves the formula~\eqref{GenerTn}, and thus Part \ref{TildeNablaPhi} of
 Proposition~\ref{prop:ladder}. \end{proof}

We can complete the proof of Theorem \ref{PoincBino}, as follows:

\begin{proof}[{\bf Proof of Theorem \ref{PoincBino}}]
We can expand function $f(k) = \sum_{j=1}^n a_j \phi_j(k)$, since the assumption that
$\sum_{k=0}^n f(k) b_{n,t}(k)= 0$ ensures that $a_0 =0 $.  Using
 the normalization term $C_{n,r}$
from Equation (\ref{eq:krawtnorm}), and the adjoint $\tn$ of
Definition \ref{def:furthermaps}, we know that
 $$\left( \tn \circ \nabla_n \right) f = \sum_{j=1}^n a_j \left( \tn \circ \nabla_n \right) \phi_j
= \sum_{j=1}^n a_j \left( \frac{ j(n-j+1)}{n^2 t(1-t)} \right) \phi_j,$$
by  Part \ref{PhiVP} of Proposition \ref{prop:ladder}. This means that
can write the RHS of Equation (\ref{eq:eigen}) as 
\begin{eqnarray*}
n t(1-t) \sum_{k=0}^n b_{n,t}(k) f(k)  \left( \tn \circ \nabla_n \right) f(k)
& = &  n t(1-t) \sum_{j=1}^n a_j^2 \frac{ j(n-j+1)}{n^2 t(1-t)} C_{n,j} \nonumber  \\
& = & \sum_{j=1}^n a_j^2 \frac{j(n-j+1)}{n} C_{n,j}  \label{eq:RHS} \\
& \geq & \sum_{j=1}^n a_j^2 C_{n,j}
\end{eqnarray*}
which is the LHS of Equation (\ref{eq:eigen}).  The inequality follows since
 $j(n-j+1)/n \geq 1$ with
equality if and only if $j = 1$ or $j=n$.
\end{proof}

\bibliography{../bibliography/papers}
\end{document}